\documentclass[12pt]{amsart}

\usepackage[latin1]{inputenc}
\usepackage{amsfonts,amssymb,amsmath,amsthm}

\usepackage{a4wide}
\usepackage[all]{xy}
\usepackage{textcomp}
\usepackage{graphicx}
\usepackage{url}
\usepackage{mathtools}
\usepackage{enumerate}
\usepackage{color}
\usepackage{enumitem}
\usepackage[normalem]{ulem}

\newif\ifdebug                                                      %
\debugtrue

\numberwithin{equation}{section}

\newtheorem{theorem}[equation]{Theorem}

\newtheorem{proposition}[equation]{Proposition}

\newtheorem{lemma}[equation]{Lemma}
\newtheorem*{lemma*}{Lemma}
\newtheorem{q}[equation]{Question}

\newtheorem{corollary}[equation]{Corollary}
\newtheorem{example}[equation]{Example}

\theoremstyle{definition}

\theoremstyle{remark}

\newtheorem{Remark*}[equation]{Remark}
\newtheorem{rmk}[equation]{Remark}

\newcommand{\wt}[1]{\widetilde{#1}}

\def    \Proj      {{\operatorname{Proj}}\, }

\DeclareMathOperator{\Sp}{Sp}

\DeclareMathOperator{\gl}{GL}
\DeclareMathOperator{\so}{SO}
\DeclareMathOperator{\un}{U}
\DeclareMathOperator{\cone}{cone}
\DeclareMathOperator{\conv}{conv}
\DeclareMathOperator{\hh}{H}
\DeclareMathOperator{\inter}{Int}

\def \Z {{\mathbb Z}}
\def \R {{\mathbb R}}
\def \Q {{\mathbb Q}}
\def \C {{\mathbb C}}
\def    \P	{{\mathbb P}}

\def \CP {{\mathbb C}{\mathbb P}}

\def \calA {{\mathcal A}}

\def \calF {{\mathcal F}}

\def    \calL   {{\mathcal L}}

\def   \mat {M_n(\mathbb{Z})}
\def    \tlambda	{{\widetilde \lambda}}
\def    \tomega{{\widetilde \omega}}
\def \tM {\wt{M}}
\def \tA {\wt{A}}
\def    \tl	{{\widetilde l}}

\def    \tx	{{\widetilde x}}

\def    \tz	{{\widetilde z}}

\title{Toric degenerations in symplectic geometry}
\author[Milena Pabiniak]{Milena Pabiniak}
\address{Mathematisches Institut, Universit\"at zu K\"oln, Weyertal 86-90, D-50931 K\"oln, Germany
\\pabiniak@math.uni-koeln.de}

\begin{document}
\maketitle
\begin{abstract}
A toric degeneration in algebraic geometry is a process where a given projective variety is being degenerated into a toric one. Then one can obtain information about the original variety via analyzing the toric one, which is a much easier object to study.
Harada and Kaveh described how one incorporates a symplectic structure into this process, providing a very useful tool for solving certain problems in symplectic geometry. 
Below we present applications of this method to questions about the Gromov width, and cohomological rigidity problems.
\end{abstract}
\section{Introduction}
Manifolds and algebraic varieties equipped with a group action are usually better understood as a presence of an action is a sign of certain symmetries. In particular, {\it toric varieties} form a very well understood class of varieties. These are varieties which contain an algebraic torus $T^n_{\C}:=(\C^*)^n$ as a dense open subset and are equipped with an action of  $T^n_\C$ which extends the usual action of $T^n_\C$ on itself.
(For more about toric varieties see, for example, \cite{CLS} and \cite{F}.)
To understand a given projective variety $X$ one can try to ``degenerate" it to a toric one, i.e., form a family of varieties with generic member $X$ and one special member some toric variety $X_0$. The varieties $X$ and $X_0$ are closely related and thus one can obtain information about $X$ by studying $X_0$. Moreover, such a degeneration gives a map from $X$ to $X_0$ which, in certain situations, is preserving some special structures $X$ and $X_0$ might be equipped with (for example: a symplectic structure).

One can use the method of toric degenerations to solve problems in symplectic geometry. In this work we discuss the following two applications:
\begin{itemize}
\item calculating lower bounds for the Gromov width, i.e., trying to find the largest ball which can be symplectically embedded into a given symplectic manifold;
\item constructing symplectomorphisms needed for a cohomological rigidity problem for symplectic toric manifolds, that is, the question of whether any two symplectic toric manifolds with isomorphic integral cohomology rings (via an isomorphism preserving the class of the symplectic form) are  symplectomorphic.
\end{itemize}
Recall that a $2n$-dimensional symplectic manifold $(M, \omega)$ equipped with an effective Hamiltonian action of an $n$-dimensional torus $T=(S^1)^n$ is called a {\it symplectic toric manifold}. 
 The action being Hamiltonian means that there exists a moment map, that is,  a $T$-invariant map  $\mu \colon M \rightarrow \textrm{Lie}(T)^*\cong \R^n$ such that for every $X \in \textrm{Lie}(T)$ it holds that $\iota_{X^\sharp} \omega=d\mu^X$ where $X^\sharp$ denotes the vector field on $M$ induced by $X$ and $\mu^X \colon M \rightarrow \R$ is defined by $\mu^X(p)=\langle \mu(p), X\rangle.$
Such a manifold can be given a complex structure interacting well with the symplectic one so that $\omega$ is a K\"ahler form and $(M,\omega)$ a K\"ahler manifold. In particular, symplectic toric manifolds are toric varieties in the sense of algebraic geometry. A theorem of Delzant states that we have a bijection 
\begin{displaymath}
\begin{array}{ccc}
 \{2n \mbox{-dim compact symplectic toric manifolds} \}
&& \{\mbox{rational, smooth polytopes
in }\R^n  \}\\
\mbox{up to  equivariant} &  \Leftrightarrow &\mbox{up to translations and }\\
\mbox{ symplectomorphisms}& & \gl(n,\Z)\mbox{ transformations}
\end{array}
\end{displaymath}
(A polytope in $\R^n$ is called rational if the directions of its edges are in $\Z^n$. It is called smooth if for every vertex the primitive vectors in the directions of edges meeting at that vertex form a $\Z$-basis for $\Z^n$.)
In this bijection, a manifold corresponds to the image of its moment map, therefore the associated polytope is often called a moment polytope or a moment image.
Not much is known about a classification of symplectic toric manifolds up to symplectomorphisms. The cohomological rigidity mentioned in the second bullet above asks if such classification might be given by the integral cohomology rings. 

In Sections \ref{sec gw} and \ref{sec coh rigidity} respectively we describe the above problems in detail and explain how one can use toric degenerations to solve problems of this type. In particular we prove (rather, outline the proofs of) the following two results, obtained in projects joint respectively with I.~Halacheva, X.~Fang and P.~Littelmann, and S.~Tolman.
\begin{theorem} \cite{HP},\cite{FLP}\label{thm gw}
Let $K$ be a compact connected simple Lie group.
The Gromov width of a coadjoint orbit $\mathcal{O}_\lambda$ through a point $\lambda$ lying on some rational line in $(Lie\, K)^*$, equipped with the Kostant--Kirillov--Souriau symplectic form, is at least
\begin{equation}\label{gw formula}
\min\{\, \left|\left\langle \lambda,\alpha^{\vee} \right\rangle \right|;\  \alpha^{\vee} \textrm{ a  coroot and }\left\langle \lambda,\alpha^{\vee} \right\rangle \neq0\}.
 \end{equation}
 \end{theorem}
\begin{theorem}\cite{PT}\label{thm coh rig}
Let $M$ and $N$ be Bott manifolds such that $\hh^*(M;\Q)$ and $\hh^*(N;\Q)$ are isomorphic to $\Q[x_1,\ldots,x_n]/\langle x_1^2, \ldots, x_n^2 \rangle$. For any ring isomorphism $F \colon \hh^*(M;\Z) \rightarrow \hh^*(N;\Z)$ sending the class $[\omega_M]$ of the symplectic form on $M$ to the class $[\omega_N]$ of the symplectic form on $N$, there exists a symplectomorphism $f \colon (N, \omega_N)\rightarrow (M,\omega_M)$ such that the map $\hh^*(f)\colon \hh^*(M;\Z) \rightarrow \hh^*(N;\Z)$ induced by $f$ on integral cohomology rings is $F$.
\end{theorem}

There are other applications of toric degenerations in symplectic geometry. For example, one can obtain information about the Ginzburg--Landau potential function on $X$ from that of $X_0$ and thus detect some non-displaceable Lagrangians of $X$ (see \cite{NNU}).

{\bf Acknowledgements.} First of all, the author would like to thank her collaborators: Xin Fang, Iva Halacheva, Peter Littelmann, and Sue Tolman. 
Results contained in this manuscript were obtained in collaboration with the above mathematicians (\cite{HP}, \cite{FLP}, \cite{PT}) and therefore all of them could also be considered as authors of this paper. (However, any remaining mistakes are due to me.)
The author also thanks the organizers of the workshops ``Interactions with Lattice Polytopes" for giving her the opportunity to participate and present her results at these workshops.

The author is supported by the DFG (Die Deutsche Forschungsgemeinschaft) grant CRC/TRR 191 ``Symplectic Structures in Geometry, Algebra and Dynamics".
\section{Toric degenerations}\label{sec toric deg}
A {\bf toric degeneration} of a projective variety $X$ is a flat family $\pi \colon \mathfrak{X} \rightarrow \C$ with generic fiber $X$ and one special fiber $ X_0=\pi^{-1}(0)$, a (not necessarily normal) toric variety. A construction of such a degeneration of a projective variety $X$, equipped with a very ample line bundle satisfying certain conditions, can be found in Anderson (\cite[Theorem 2]{A}).
\begin{example}
Using the Pl\"ucker embedding,\footnote{Recall that the Pl\"ucker embedding sends a Grassmannian spanned by vectors $v,w \in \C^4$ to a point $[x_{12}:\ldots:x_{34}]\in \C\P^5$ with $x_{ij}$ equal to the determinant of the $2\times 2$ minor of $[v^T,w^T]$ spanned by rows $i$ and $j$.}
 view $X=Gr(2,\C^4)$, the Grassmannian of $2$-planes in $\C^4$, as a subset of $\C\P^5$ with coordinates $\{x_{ij};\ 1\leq i<j \leq 4\}$, consisting of points satisfying
$$x_{12}x_{34}-x_{13}x_{24}+x_{14}x_{23}=0.$$
Consider the subset
$ \mathfrak{X} \subset \C\P^5 \times \C$ consisting of points satisfying 
$$x_{12}x_{34}-x_{13}x_{24}+tx_{14}x_{23}=0,$$
where $t$ denotes the coordinate in $\C$. Let $\pi \colon \mathfrak{X} \rightarrow \C$ be the restriction to $\mathfrak{X} $ of the projection onto the second factor. This family constitutes a toric degeneration of $Gr(2,\C^4)$. Clearly $\pi^{-1}(1)$ is $Gr(2,\C^4)$. Moreover, performing a change of coordinates, one can show that $\pi^{-1}(t)$ for $t\neq 0$ is also bihomolomorphic to $Gr(2,\C^4)$. The central fiber, $\pi^{-1}(0)$, is described by the binomial ideal $\langle x_{12}x_{34}-x_{13}x_{24} \rangle $ and thus is a toric variety.
\end{example}
Harada and Kaveh in \cite{HK} enriched the construction of Anderson by incorporating a symplectic structure.
They start with a smooth projective variety $X$, of complex dimension $n$, equipped with a very ample line bundle $\calL$, with some fixed Hermitian structure. Let $L:=\hh^0(X, \mathcal{L})$ denote the vector space of holomorphic sections, $\Phi_\calL \colon X \rightarrow \P (L^*)$ the Kodaira embedding and $\omega=\Phi_\calL^*(\omega_{FS})$ the pull back of the Fubini--Study form, i.e., of the standard symplectic structure on complex projective spaces. Then $(X,\omega)$ is a K\"ahler manifold.
With this data they construct (under certain assumptions) not only a flat family $\pi \colon \mathfrak{X} \rightarrow \C$ but also a K\"ahler structure $\tilde{\omega}$ on (the smooth part of) $\mathfrak{X}$ so that $(\pi^{-1}(1), \tilde{\omega}_{|\pi^{-1}(1)})$ is symplectomorphic to $(X,\omega)$.
Moreover, the special fiber $X_0=\pi^{-1}(0)$ inherits a $2$-form, the restriction of $\tilde{\omega}$, defined on its smooth part $U_0:=(X_0)_{\textrm{smooth}}$, and thus it also obtains a divisor. If $X_0$ is normal, then the polytope associated to $X_0$ and this divisor by the usual procedure of toric algebraic geometry (see, for example, Chapter 4 of \cite{CLS}) is the closure of the moment image of the (non-compact) symplectic toric manifold $(U_0,\tilde{\omega}_{|U_0}).$ As we will see,
this polytope can be computed by analyzing the behaviour of the holomorphic sections of $\calL$.
Here are more details about this procedure.

Denote by $L^m$ the image of the $\textrm{span }\langle f_1\cdot \ldots \cdot f_m\,; \ f_i \in L \rangle$ in $\hh^0(X, \mathcal{L}^{\otimes m})$ and by 
$R=\C[X]=\oplus_{m\geq 0}\,L^m$ the homogeneous coordinate ring of $X$ with respect to the embedding $\Phi_\calL$. 
An important ingredient of the construction is a choice of a {\it valuation with one dimensional leaves}, $\nu \colon \C(X)\setminus \{0\} \rightarrow \Z^n$, from the ring $\C(X)$ of rational functions on $X$. 
A precise definition of a general valuation can be found, for example, in \cite[Definition 3.1]{HK}. 
In this paper we only use valuations induced by a flag of subvarieties and a special case of these, called {\it lowest/highest term valuations associated to a coordinate system}.

\begin{example}[Lowest/highest term valuations of a coordinate system] \cite[Example 3.2]{HK}\label{eg lh val}
Fix a (smooth) point $p \in X$ and let $(u_1, \ldots, u_n)$ be a system of coordinates in a neighborhood of $p$, meaning that $u_1,\ldots, u_n$ are regular functions at $p$, vanishing at $p$, and such that their differentials $du_1,\ldots, du_n$ are linearly independent at $p$. Then any regular function at $p$ can be represented as  a power series $\sum_{\alpha \in \Z_{\geq 0}^n}\,c_\alpha u^\alpha$. Here by $ u^\alpha$, with $\alpha=(\alpha_1, \ldots, \alpha_n)\in \Z_{\geq 0}^n$, we mean $u_1^{\alpha_1}\cdot \ldots \cdot u_n^{\alpha_n}$. Choose and fix a total order $>$ on $\Z^n$ respecting the addition, for example the lexicographic order. Define a map $\nu$ from the set of functions regular at $p$ to $\Z^n$ by 
$$\nu\, \big(\sum_{\alpha \in \Z_{\geq 0}^n}\,c_\alpha u^\alpha \big)=\min\{\alpha;\ c_\alpha \neq 0\},$$
and extend it to $\C(X)\setminus \{0\}$ by setting $\nu (f/g)=\nu(f)-\nu(g)$. Then $\nu$ is a valuation with one dimensional leaves, called a \emph{lowest term valuation}. If one uses $\max$ instead of $\min$ in the definition of $\nu$, one obtains a \emph{highest term valuation}.
\end{example}
\begin{example}[Valuations induced by a flag of subvarieties] \cite[Example 3.3]{HK}\label{eg flag val} 
Take a flag of normal subvarieties (called a Parshin point) of $X$
$$\{p\}=Y_n \subset \ldots \subset Y_0=X,$$
with $\dim_\C (Y_k)=n-k$ and $Y_k$ non-singular along $Y_{k+1}$. 
By the non-singularity assumption there exists a collection of rational functions $u_1,\ldots,u_n$ on $X$ such that $u_{k|Y_{k-1}}$ is a rational function on $Y_{k-1}$ which is not identically zero and which has a zero of first order on $Y_{k}$. 
Then the lowest term valuation with respect to the lexicographic order can alternatively be described in the following way: for any $f \in \C(X)$, $f \neq 0$, the valuation $v(f)=(k_1,\ldots,k_n)$ where
$k_1$ is the order of vanishing of $f$ on $Y_1$, $k_2$ is the order of vanishing of $f_1:=(u_1^{-k_1} f)_{|Y_1}$ on $Y_2$, etc. 
\end{example}

Given such $X$, $\calL$, and $\nu$ we form a semigroup $S=S(\nu, \calL)$,
 in the following way. Fix a non-zero element $h \in L$ and use it to identify\label{identification} $L$ with a subspace of $\C(X)$ by mapping $f \in L$ to $f/h \in \C(X)$. 
 Similarly identify 
 $L^m$ with a subspace of $\C(X)$ by sending $f \in L^m$ to $f/h^m \in \C(X)$. As any valuation satisfies $\nu(fg)=\nu(f)+\nu(g)$, the set
$$S=S(\nu,\calL)=\bigcup _{m\geq 0} \{ (m, \nu(f/h^m))\,|\,f \in L^m \setminus \{0\}\,\}$$
is a semigroup with identity (i.e. a monoid). 
If $S$ is finitely generated, one can construct a toric degeneration whose special fiber is a toric variety $\Proj \C[S]$ (which is normal if $S$ is saturated).
Moreover we obtain an Okounkov body 
$$ \Delta=\Delta(S)=\overline{\conv \big( \bigcup_{m>0} \{x/m\,|\, (m,x) \in S\}\big)}\subset \R^n.$$
Note that if $S$ is finitely generated, then $\Delta$ is a rational convex polytope.
The toric variety corresponding to $\Delta$ is the normalization of $\Proj \C[S]$.

In the following theorem we rephrase several results from \cite{HK}.
\begin{theorem}\cite{HK}\label{from hk}
Let $\calL$ be a very ample Hermitian line bundle on a
smooth $n$-dimensional projective variety $X$ and $\omega=\Phi_\calL^*(\omega_{FS})$ the induced symplectic form.
Let $\nu \colon \C(X) \setminus \{0\}\rightarrow \Z^n$ be a valuation with one dimensional leaves, and such that
the associated semigroup $S$ is finitely generated.
Then 
\begin{itemize}
\item There exists a toric degeneration $\pi \colon \mathfrak{X} \rightarrow \C$ with generic fiber $X$ and special fiber $X_0 := \Proj \C[S]$, and a K\"ahler structure $\tilde{\omega}$ on (the smooth part of) $\mathfrak{X}$ such that 
$(\pi^{-1}(1), \tilde{\omega}_{|\pi^{-1}(1)})$ is symplectomorphic to $(X,\omega)$
and the closure of the moment image of symplectic toric manifold 
$(U_0, \tilde{\omega}_{|U_0})$, where $U_0:=(X_0)_{\textrm{smooth}}$, is the Okounkov body $\Delta (S)$. 
The set $U_0$ contains the preimage of the interior of $\Delta(S)$.
\item Moreover, there exists 
a surjective continuous map $\phi \colon X \to X_0$ 
that restricts
to a symplectomorphism from $(\phi^{-1}(U_0), \omega)$ to $(U_0, \tilde{\omega}_{|U_0})$.
\end{itemize}
In particular, if $X_0=\Proj \C[S]$ is smooth (thus also normal), then $\phi^{-1}(U_0)=X$ and therefore $\phi$ provides a symplectomorphism between $(X,\omega)$ and the symplectic toric manifold $(X_{\Delta(S)}, \omega_{\Delta(S)})$ associated to $\Delta(S)$ via Delzant's construction. 
\end{theorem}

Checking whether $S$ is finitely generated is a very difficult problem.
However, it was observed by Kaveh in \cite{K} that even if $S$ is not finitely generated one can still form a (not flat) family with generic fiber $X$ and special fiber $(\C^*)^n$. Even though such a construction provides much less information about $X$, it still suffices for the purpose of finding lower bounds on the Gromov width. We describe this idea in Section \ref{sec gw}.

\section{Gromov width}\label{sec gw}
The {\it Gromov width} of a $2n$-dimensional symplectic manifold $(X,\omega)$
is the supremum of the set of the positive real numbers $a$ such that the ball of capacity $a$ (or, equivalently, of radius $\sqrt{\frac{a}{\pi}}$),
$$ B^{2n}_a =B^{2n} \Big(\sqrt{\frac{a}{\pi}}\Big)= \big \{ (x_1,y_1,\ldots,x_n,y_n) \in  \R^{2n} \ \Big | \ \pi \sum_{i=1}^n (x_i^2+y_i^2) < a \big 
\} \subset  (\R^{2n}, \omega_{st}), $$
 can be symplectically
embedded in $(X,\omega)$. Here $\omega_{st}=\sum_j\,dx_j \wedge dy_j$ denotes the standard symplectic form on $\R^{2n}$. 
This question was motivated by the Gromov non-squeezing theorem which states that a ball $B^{2n}(r) \subset (\R^{2n}, \omega_{st})$ cannot be symplectically embedded into $B^2(R) \times \R^{2n-2} \subset (\R^{2n}, \omega_{st})$ unless $r\leq R$. 

$J$-holomorphic curves give obstructions to ball embeddings, while Hamiltonian torus actions can lead to constructions of such embeddings (by extending a Darboux chart using the flow of the vector field induced by the action). 

This is why toric degenerations provide a useful tool for finding lower bounds on the Gromov width. Given a toric degeneration of $(X,\omega)$, as described in Theorem \ref{from hk}, one can use the toric action on $X_0$ to construct embeddings of balls into a smooth symplectic toric manifold $(U_0, \tilde{\omega}_{|U_0})$, where $U_0=(X_0)_{\textrm{smooth}}$. Postcomposing such embedding with the symplectomorphism $\phi^{-1}$ produces a symplectic embedding into $(X,\omega)$.

Moreover, many embeddings of balls into symplectic toric manifolds can be read off from the associated (by the Delzant classification theorem) polytope. Identify the dual of the Lie algebra of the compact torus $T$ with the Euclidean space using the convention that $S^1=\R/\Z$, i.e., the lattice of $\mathfrak{t}^*$ is mapped to $\Z^{\dim T} \subset \R^{\dim T}$. With this convention, the moment map for the standard $(S^1)^n$ action on $(\R^{2n}, \omega_{st})$ maps $ B^{2n}_a$ onto an $n$-dimensional simplex of size $a$, closed on $n$ sides
\begin{equation}\label{simplex}
\mathfrak S^n(a):=\{(x_1,\ldots,x_n) \in \R^n|\ 0\leq x_j< a,\  \sum_{j=1}^n x_j< a\}.
\end{equation}
Moreover, if the moment image contains an open simplex of size $a$, then for any $\varepsilon >0$ a ball of capacity $a - \varepsilon$ can be embedded into the given symplectic toric manifold.
\begin{proposition}\cite[Proposition 1.3]{LuSC}\cite[Proposition 2.5]{P}\label{embedding}
For any connected, proper (not necessarily compact) symplectic toric manifold  $U$ of dimension $2n$, with a moment map $\mu$, the Gromov width of $U$ is at least
$$\sup \{a>0\,|\, \exists \; \Psi \in \gl(n,\Z), x \in
\R^n,\textrm{ such that }
\Psi (\inter \mathfrak S^n(a))+\,x \subset \mu(U) \}.
$$
\end{proposition}
The appearance of $\Psi$ and $x$ comes from the facts that the identification $\mathfrak{t}^*\cong \R^{\dim \ T}$ depends on a splitting of $T$ into $(\dim \ T)$ circles, and that a translation of a moment map also provides a moment map.

The above results lead to the following method for finding lower bounds on the Gromov width.
\begin{corollary}\label{cor gw tool}
Let $X$ be a smooth projective variety of complex dimension $n$, $\calL$ an ample line bundle on $X$, and $\omega=\Phi_\calL^*(\omega_{FS}) \in \hh^2(X; \Z)$ an integral K\"ahler form obtained using the Kodaira embedding $\Phi_\calL \colon X \rightarrow \P (L^*)$. Suppose that there exists a valuation $\nu$ giving a finitely generated and saturated semigroup $S=S(\nu, \calL)$.
 Let $\Delta$ be the associated Okounkov body. 
The Gromov width of $(X,\omega)$ is at least 
$$\sup \{a>0\,|\, \exists \; \Psi \in \gl(n,\Z), x \in
\R^n,\textrm{ such that }
\Psi (\inter\mathfrak S^n(a))+\,x \subset \Delta \}.
$$ 
\end{corollary}
\begin{proof}
By the result of \cite{HK} cited above as Theorem \ref{from hk}, there exists a toric degeneration of $(X, \omega)$ to a normal toric variety $X_0=\Proj \C[S]$. The subset $U:=\phi^{-1}(U_0)$ of $X$ inherits a toric action whose moment image contains $\inter\, \Delta$, the interior of $\Delta$ (recall that a moment map sends singular points of a toric variety to the boundary of the moment polytope). The corollary follows from Proposition \ref{embedding}.
\end{proof}
In fact one does not need $S$ to be saturated. The same corollary holds even if $X_0$ is not a normal toric variety. This is because a normalization map for $X_0$ induces a biholomorphism between $(X_0)_{\textrm{smooth}}$ and an appropriate subset of the normalization of $X_0$.

It is, however, necessary that $S$ is finitely generated for a toric degeneration to exist. Otherwise one can still form a family of manifolds, but one cannot guarantee for this family to be flat, and thus $X$ and $X_0$ are no longer so strongly related. As we already mentioned, Kaveh in \cite{K} observed that such a (not necessarily flat) family, with $X_0=(\C^*)^n$, still provides information about the Gromov width of $(X,\omega)$.
To state this result we need additional notation. In the notation of Section \ref{sec toric deg}, for any $m\in\Z_{>0}$ let 
$$\calA_m:=\{\nu(f/h^m)\,|\,f \in L^m \setminus \{0\}\,\}\subset \Z^n, \quad \Delta_m=\frac{1}{m} \, \conv (\calA_m).$$
Note that $\Delta=\overline{\cup_{m>0} \Delta_m}$.
Fix $m$ and let $r=r_m$ denote the number of elements in $\calA_m=\{\beta_1, \ldots, \beta_r\}$.
From these data we form a symplectic form, $\omega_{m}$, on $(\C^*)^n$ using a standard procedure: $\omega_{m}$ is the pull back of the Fubini--Study form on $\C\P^{r-1}$ via the map
$\Psi_m \colon (\C^*)^n \rightarrow \C\P^{r-1}$, $u \mapsto (u^{\beta_1}c_1, \ldots,u^{\beta_r}c_r )$, where  $c=[(c_1,\ldots,c_r)]$ is some element in $\C\P^{r-1}$ with all $c_i\neq 0$. (In \cite{K} the elements $c_i$ come from coefficients of leading terms of elements in an appropriately chosen basis of $L^m$. One also needs that the differences of elements in $\calA_m$ span $\Z^n$ which, by \cite[Remark 5.6]{K}, is always true for lowest term valuations.)

Kaveh proved that
\begin{itemize}
\item For every $m>0$ there exists an open subset $U \subset X$ such that $(U, \omega)$ is symplectomorphic to $((\C^*)^n,  \frac{1}{m}\omega_{m})$ (\cite[Theorem 10.5]{K}).
\item The Gromov width of $((\C^*)^n,\frac{1}{m}\omega_{m})$ is at least $R_m$, where $R_m$ is the size of the largest open simplex that fits in the interior of $\Delta_m=\frac{1}{m}\conv\,(\calA_m)$(\cite[Corollary 12.3]{K}).
\end{itemize}
This leads to the following corollary.
\begin{corollary}\cite[Corollary 12.4]{K}\label{cor k gw}
Let $X$ be a smooth projective variety of dimension $n$, $\calL$ an ample line bundle on $X$, and  $\omega=\Phi_\calL^*(\omega_{FS}) \in H^2(X; \Z)$ an integral K\"ahler form. Let $\nu$ be a lowest term valuation on $\C(X)$, with values in $\Z^n$, and $\Delta$ the associated Okounkov body.
The Gromov width of $(X,\omega)$ is at least $R$, where $R$ is the size of the largest open simplex that fits in the interior of $\Delta$.
\end{corollary}
\subsection{Results about coadjoint orbits.}
The methods for finding the Gromov width described in Corollaries \ref{cor gw tool} and \ref{cor k gw} have been used in \cite{HP} and \cite{FLP} for coadjoint orbits of compact Lie groups.

Recall that given a compact Lie group $K$ each orbit 
$\mathcal{O}\subset \mathfrak k^*:=(\textrm{Lie}\,K)^*$ of the coadjoint action of $K$ on  $\mathfrak k^*$
is naturally a symplectic manifold. Namely it can be equipped with the Kostant--Kirillov--Souriau symplectic form $\omega^{KKS}$ defined by:
$$\omega^{KKS}_{\xi}(X^\#,Y^\#)=\langle \xi, [X,Y]\rangle,\;\;\;\xi \in \mathcal{O} \subset \mathfrak k^*,\;X,Y \in \mathfrak k,$$
where $X^\#,Y^\#$ are the vector fields on $\mathfrak k^*$ induced by $X,Y \in \mathfrak k$ via the coadjoint action of $K$. 
Coadjoint orbits are in bijection with points in a positive Weyl chamber as every coadjoint orbit intersects such a chamber in a single point. 
An orbit is called generic (resp. degenerate) if this intersection point is an interior point of the chamber (resp. a boundary point). An orbit passing through a point $\lambda$ in a positive Weyl chamber will be denoted by $\mathcal{O}_{\lambda}$.
For example, when $K=\un(n,\C)$ is the unitary group, a coadjoint orbit can be identified with the set of Hermitian matrices with a fixed set of eigenvalues. The orbit is generic if all eigenvalues are different, and in this case it is diffeomorphic to the manifold of complete flags in $\C^n$.

It has been conjectured that the Gromov width of $(\mathcal{O}_{\lambda}, \omega^{KKS})$ should be given by the following neat formula, expressed entirely in the Lie-theoretical language
$$\min\{\, \left|\left\langle \lambda,\alpha^{\vee} \right\rangle \right|;\  \alpha^{\vee} \textrm{ a  coroot and }\left\langle \lambda,\alpha^{\vee} \right\rangle \neq0\}.$$
For example, as $\{e_{ii}-e_{jj};\, i\neq j\}$ forms a root system for the unitary group $\un (n, \C)$, the Gromov width of its coadjoint orbit $\mathcal{O}_\lambda$ passing through a point $\lambda=diag\,(\lambda_1,\ldots,\lambda_n)\in \mathfrak{u}(n)^*$, lying on some rational line, is equal to 
$$\min \{|\lambda_i-\lambda_j|;\ i,j \in \{1,\ldots, n\},\ \lambda_i\neq \lambda_j\}.$$
Here we identified both $\mathfrak{u}(n)$ and $\mathfrak{u}(n)^*$ with the set of $n \times n$ Hermitian matrices.

This conjecture was motivated by the computation of the Gromov width of complex Grassmannians, i.e., degenerate coadjoint orbits of $\un(n,\C)$, done by Karshon and Tolman in \cite{KT}, and independently by Lu in \cite{LuGW}. Later, using holomorphic techniques, Zoghi (for generic indecomposable\footnote{A coadjoint orbit through a point $\lambda$ in the interior of a chosen positive Weyl chamber is called  indecomposable in \cite{Z} if there exists a simple positive root $\alpha$ such that for any positive root $\alpha'$ there exists a positive integer $k$ such that $\langle \lambda, \alpha' \rangle=k \langle \lambda, \alpha \rangle$. }  orbits of $\un(n,\C)$, \cite{Z}) and Caviedes (for any coadjoint orbit, \cite{CC}) showed that the above formula provides an upper bound for the Gromov width. The fact that this formula also provides a lower bound was proved using explicit Hamiltonian torus actions by Zoghi (for generic indecomposable orbits of $\un(n,\C)$ using the standard action of the maximal torus, \cite{Z}), Lane (for generic orbits of the exceptional group $\textrm{G}_2$, \cite{La}), and the author (for $\un(n,\C)$, $\so(2n,\C)$ and $\so(2n+1,\C)$ orbits
\footnote{The result about $\so(2n+1,\C)$ holds only for orbits satisfying a mild technical condition: the point $\lambda$ of intersection of the orbit and a chosen positive Weyl chamber should not belong to a certain subset of one wall of the chamber; see \cite{P} for more details. In particular, all generic orbits satisfy this condition.}
 using the Gelfand--Tsetlin torus action, \cite{P}).

\subsection{A sketch of the proof of Theorem \ref{thm gw}}
The first usage of toric degenerations in Gromov width problems appeared in the work \cite{HP} of Halacheva and the author, where the generic orbits of the symplectic group $\Sp(n)=\un(n,\mathbb{H})$ are considered. Then it was used in \cite{FLP} to prove that the formula \eqref{gw formula} is a lower bound for the Gromov width of any coadjoint orbit of any compact connected simple Lie group $K$, passing through a point in the Weyl chamber lying on some rational line, i.e., to prove Theorem \ref{thm gw}.

The rationality assumption comes from the fact that the toric degeneration method can be applied only to the orbits passing through an integral point $\lambda$ of a positive Weyl chamber, i.e., in the language of representation theory, through a dominant weight. 

Let $G$ be a simply connected simple complex algebraic group and $K\subset G$ be its maximal compact subgroup.
With a dominant weight $\lambda$ one can associate an irreducible representation $V(\lambda)$ of $G$ of highest weight $\lambda$. Let 
$\mathbb C_{v_\lambda}$ be the highest weight line and $P=P_{\lambda}$ be the 
normalizer in $G$ of this line. Then the coadjoint orbit $\mathcal{O}_{\lambda}$ of $K$ is diffeomorphic to $G/P$ (and to $K/K \cap P$) and there exists a very ample line bundle $\mathcal L_\lambda$ on $G/P$ such that the pull back of the Fubini--Study form on the projective space 
$\mathbb P(\mathrm{H}^0(G/P,\mathcal L_\lambda)^*)=\mathbb P(V(\lambda))$ via the Kodaira embedding $ G/P \rightarrow \mathbb P(\mathrm{H}^0(G/P,\mathcal L_\lambda)^*)$ is exactly the Kostant--Kirillov--Souriau symplectic form $\omega^{KKS}$ on $\mathcal{O}_{\lambda}$ (see for example \cite[Remark 5.5]{CC}). 
Thus for integral $\lambda$'s one can try to construct toric degenerations of projective variety $G/P$ with line bundle $\mathcal L_\lambda$ and obtain some lower bounds for the Gromov width of the orbit $\mathcal{O}_{\lambda}$. Rescaling of symplectic forms allows to extend such a result to orbits $\mathcal{O}_{a\lambda}$, for any $a \in \R_{>0}$.

It remains to discuss how one can construct desired toric degenerations.

A great advantage of working with coadjoint orbits of a complex algebraic group $G$ is that a lot of information can be obtained from studying representations of $G$. This leads to a beautiful interplay between symplectic geometry and representation theory. 
A reduced decomposition of the longest word in the Weyl group, $\underline{w}_0=s_{i_{\alpha_1}} \cdot \ldots \cdot s_{i_{\alpha_N}}$ provides the following items (defined below) related in an interesting way:
\begin{itemize}
\item a valuation $\nu_{\underline{w}_0}$,
\item a string parameterization of a crystal basis of $V_{\lambda}^*$.
\end{itemize}

We continue to denote by $\lambda$ a dominant weight and by $V_{\lambda}$ the finite dimensional irreducible representation of $G$ with highest weight $\lambda$. Let $V_{\lambda}^*$ denote the dual representation. 
One often seeks for a basis of $V_{\lambda}^*$ consisting of elements which behave nicely under the action of Kashiwara operators. A crystal basis is a basis whose elements are permuted under the Kashiwara operators. Given a crystal basis one can form a crystal graph of a given representation: vertices are elements of the crystal basis and $\{0\}$, and edges are labelled by simple roots and correspond to the action of Kashiwara operators. There are (noncanonical) ways of embedding such graph into $\R^{N}$, $N=\dim_\C G/P$. A reduced decomposition of the longest word in the Weyl group (into a composition of reflections with respect to simple roots), $\underline{w}_0=s_{\alpha_1} \cdot \ldots \cdot s_{\alpha_N}$, provides a way of assigning to each vertex of the crystal graph a string of $N$ integers (string parametrization), and thus gives such an embedding. A convex hull of the image of string parametrization is called a string polytope. It depends on $\lambda$ and also on the chosen decomposition $\underline{w}_0$. String polytopes have been extensively studied in representation theory.

Moreover, a reduced decomposition $\underline{w}_0=s_{i_{\alpha_1}} \cdot \ldots \cdot s_{i_{\alpha_N}}$ defines a sequence of Schubert subvarieties
$$[P]=Y_{N} \subset \ldots \subset Y_{0}= G/P,$$
where $Y_j$ denotes the Schubert variety corresponding to the element $s_{i_{\alpha_{j+1}}} \cdot \ldots \cdot s_{i_{\alpha_N}}$ of the Weyl group.
We denote by $\nu_{\underline{w}_0}$ the highest term valuation associated with this flag of subvarieties.

A theorem of Kaveh relates these two objects.
\begin{theorem}\cite{Kcrbas}
The string parametrization for $V_\lambda^*=\hh^0(G/P, \calL_\lambda)$ obtained using the reduced decomposition $\underline{w}_0$ is the restriction of the valuation $\nu_{\underline{w}_0}$ and thus the corresponding string polytope is the Okounkov body $\Delta (\nu_{\underline{w}_0})$.
\end{theorem}

Explicit descriptions of string polytopes for classical Lie groups and some well-chosen reduced decompositions of the longest words were presented in the work of Littelmann \cite{Li}. With a bit of work, one can show that the string polytope for $V_\lambda^*$ with $G=\Sp(2n,\C)$ the symplectic group (with maximal compact subgroup $K=\Sp(n)=\textrm{U}(n,\mathbb{H})$), described in \cite{Li}, contains a simplex of size prescribed by \eqref{gw formula}. Then, the result of Kaveh, \cite{Kcrbas}, quoted above together with Corollary \ref{cor gw tool} prove that the Gromov width of $\Sp(n)$ coadjoint orbit $(\mathcal{O}_\lambda,\omega^{KKS})$ is at least equal to the value prescribed by \eqref{gw formula}, i.e., proves Theorem \ref{thm gw} for the case of the symplectic group. This is exactly the argument used in \cite{HP}.

Similar method could be applied for other classical Lie groups. However, one would need to consider each type separately, as the descriptions of string polytopes contained in \cite{Li} depend on reduced decompositions which are different for different group types.

To obtain a unified proof which works for all group types, in \cite{FLP} we use lowest term valuations arising from a system of parameters induced by an enumeration $\{\beta_1, \ldots, \beta_N\}$ of certain positive roots. In \cite{FFL} the authors gave a representation-theoretic description of the associated semigroup, also in the cases where this enumeration does not come from a reduced decomposition of the longest word. Unfortunately it might be very difficult to show that this semigroup is finitely generated and to find an explicit description of the associated Okounkov body. However, in the case when the enumeration is a good ordering in the sense of \cite{FFL}, building on the results from \cite{FFL} one can at least show that the associated Okounkov body contains a simplex of size prescribed by \eqref{gw formula}. Then, using the result of Kaveh \cite{K} cited here as Corollary \ref{cor k gw} (which does not require the semigroup to be finitely generated), one proves Theorem \ref{thm gw}. The details of this argument are presented in \cite{FLP}.
\section{Cohomological rigidity}\label{sec coh rigidity}
The following section is based on a project joint with Sue Tolman, \cite{PT}.

Cohomological rigidity problems are problems where one tries to determine whether integral cohomology ring can distinguish between manifolds of a certain family, and whether all isomorphisms between integral cohomology rings are induced by maps (homeomorphisms or diffeomorphisms, depending on the setting) between manifolds.
Integral cohomology ring is too weak to distinguish a homeomorphism type of a manifold. However, by a result of Freedman, it classifies (up to a homeomorphism) all closed, smooth, simply connected $4$-manifolds. 
Masuda and Suh posed question of whether the cohomological rigidity holds for the family of symplectic toric manifolds. 
The question was studied by its authors, Choi, and Panov. No counterexample was found and partial positive results were proved.
(Interested readers are encouraged to consult the nice survey \cite{CMSsurvey} and references therein.)
Due to the presence of symplectic structure, it seems natural to consider the following symplectic variant of the above question.
\begin{q}(Symplectic cohomological rigidity for symplectic toric manifolds)
\begin{itemize}
\item (weak) Are any two symplectic toric manifolds $(M, \omega_M)$ and $(N, \omega_N)$ necessarily symplectomorphic whenever there exists an isomorphism $F \colon \hh^*(M;\Z) \rightarrow \hh^*(N;\Z)$ sending the class $[\omega_M]$ to $[\omega_N]$?
\item (strong) Is any such
 isomorphism $F \colon \hh^*(M;\Z) \rightarrow \hh^*(N;\Z)$ induced by a symplectomorphism?
\end{itemize}
\end{q}
Sue Tolman and the author, in \cite{PT}, prove that weak and strong symplectic cohomological rigidity hold for the family of Bott manifolds with rational cohomology ring isomorphic to that of a product of copies of $\CP^1$. Bott manifolds can be viewed as higher dimensional generalizations of Hirzebruch surfaces discussed in the example below. For definition see Section \ref{sec bott mfd}.
\begin{rmk}
Strong (not symplectic) cohomological rigidity, with diffeomorphisms, was already proved for this family by Choi and Masuda in \cite{CM}. Their diffeomorphisms usually do not preserve the complex structure. If they had, then our result would be an immediate consequence of theirs. Indeed, if $f \colon N \rightarrow M$ is a biholomorphism inducing $F$, then $\omega_N$ and $f^*(\omega_M)$ are both K\"ahler forms on $N$, defining the same cohomology class in $\hh^*(N;\Z)$, and thus in this case $(N,\omega_N)$ and $(N,f^*(\omega_M))$ are symplectomorphic by Moser's trick.
\end{rmk}
\begin{example}(Hirzebruch surfaces)\label{eg h}
Hirzebruch surfaces are $\CP^1$ bundles over $\CP^1$. As complex manifolds they are classified by integers (encoding the twisting of the bundle): for each $A \in \Z$ we denote by $\mathcal{H}_{-A}$ the bundle $\P (\mathcal{O}(A) \oplus \mathcal{O}(0)) \rightarrow \CP^1$. In particular, $\mathcal{H}_0=\CP^1 \times \CP^1$. They can be equipped with a symplectic (even K\"ahler) structure and a toric action. A polytope corresponding to $\mathcal{H}_{-A}$ in Delzant's classification is (up to $\gl(2,\Z)$ action) a trapezoid with outward normals $(-1,0)$,$(0,-1)$,$(1,0)$,$(A,1)$. The lengths of the edges of this trapezoid depend on the chosen symplectic structure and can be encoded as $\lambda=(\lambda_1,\lambda_2)\in (\R_{>0})^2$. 
We denote by $(\mathcal{H}_{-A},\omega_\lambda )$ the symplectic toric manifold corresponding to the trapezoid  $\Delta (A, \lambda):=\conv ((0,0),(\lambda_1,0), (\lambda_1,\lambda_2-A\lambda_1),(0,\lambda_2))$. For example, Figure \ref{fig Hirzebruch} presents $(\mathcal{H}_0,\omega_{(1,3)} )$ and $(\mathcal{H}_{-2},\omega_{(1,5)} )$.
\begin{figure}[h]
\includegraphics[scale=0.5]{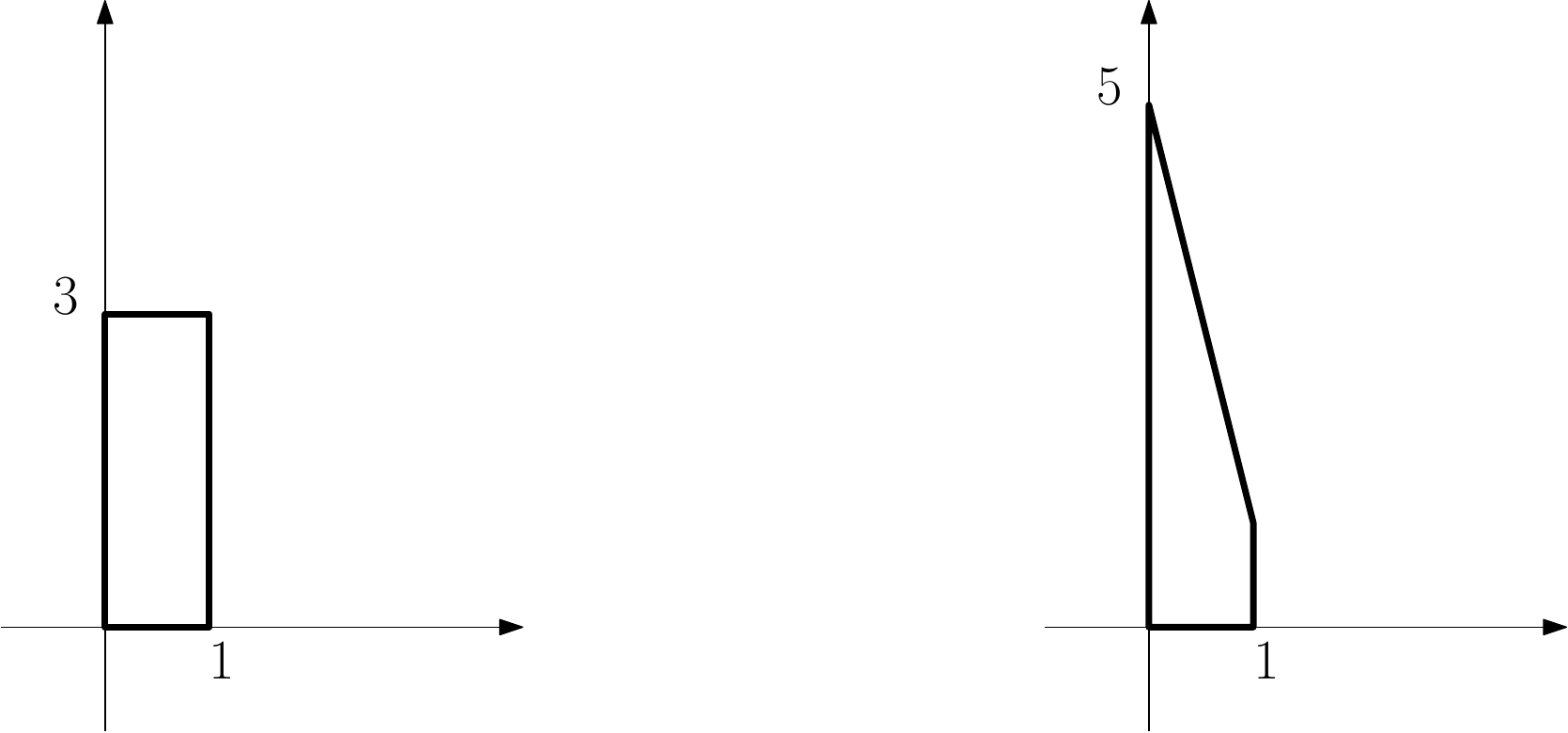}
\caption{Hirzebruch surfaces $(\mathcal{H}_0,\omega_{(1,3)} )$ and $(\mathcal{H}_{-2},\omega_{(1,5)} )$.}
\label{fig Hirzebruch}
\end{figure}

It was observed by Hirzebruch that $\mathcal{H}_{-A}$ and $\mathcal{H}_{-\tA}$ are diffeomorphic if and only if $A \cong \tA \mod 2$. 
Moreover, the symplectic toric manifolds $(\mathcal{H}_{-A},\omega_\lambda )$ and $(\mathcal{H}_{-\tA},\omega_\tlambda )$ are (not equivariantly) symplectomorphic if and only if $A \cong \tA \mod 2$ and the widths and the areas of the associated polytopes agree, i.e., $\lambda_1=\tlambda_1$ and 
$ \lambda_2-\frac{1}{2}A\lambda_1 =\tlambda_2-\frac{1}{2}\tA\tlambda_1$. 
For example, the manifolds presented on Figure \ref{fig Hirzebruch} are symplectomorphic.
The cohomology ring can be presented as
$$\hh^*(\mathcal{H}_{-A}; \Z)=\Z[x_1,x_2]/\langle x_2^2, x_1^2+Ax_1x_2\rangle,$$
with $[\omega_\lambda]=\lambda_1 x_1 +\lambda_2 x_2$.
If $A \cong \tA \mod 2$, then the isomorphism $\Z[x_1,x_2] \rightarrow \Z[\tx_1,\tx_2]$ defined by  $x_1 \mapsto \tx_1 +\frac{1}{2}(\tA-A)\tx_2$, $x_2 \mapsto \tx_2$ descends to an isomorphism between $\hh^*(\mathcal{H}_{-A}; \Z)$ and $\hh^*(\mathcal{H}_{-\tA}; \Z)$. Note that this isomorphism sends $[\omega_\lambda]=\lambda_1 x_1 +\lambda_2 x_2$ to $\lambda_1 \tx_1+(\lambda_2 + \frac{\tA-A}{2}\lambda_1)\,\tx_2 $ which is equal to $[\omega_\tlambda]$ if and only if $\lambda_1=\tlambda_1$ and 
 $\lambda_2-\frac{A}{2}\lambda_1 =\tlambda_2-\frac{\tA}{2}\tlambda_1$.
 Therefore, for Hirzebruch surfaces (weak) symplectic cohomological rigidity holds.
\end{example}

To approach the symplectic cohomological rigidity problem one needs a good method of constructing symplectomorphisms. Here is where toric degenerations come into play. By Theorem \ref{from hk} a toric degeneration whose central fiber $\Proj\, \C[S]$ is smooth produces a symplectomorphism between the symplectic manifold one started with and the central fiber. The main difficulty in this method of constructing symplectomorphisms lies in finding toric degenerations with smooth central fibers.

A great advantage of working with toric manifolds is that the sections of their line bundles are well understood and one can form very concrete constructions of toric degenerations. 

\subsection{Toric degenerations for symplectic toric manifolds}
Let $(X_P,\omega_P)$ be a symplectic toric manifold with $\omega_P \in \hh^2(M; \Z)$, corresponding to a polytope $P \subset \R^n$ via Delzant's construction. Then $P$ is an integral polytope (i.e. with vertices in $\Z^n$) and there exists a very ample line bundle $\calL$ over $X_P$ inducing $\omega_P$. In this situation a basis of the space of holomorphic sections of $\calL$ can be identified with the integral points of $P$
(\cite{D}, see also \cite{H}).
Without loss of generality we can assume that in a neighborhood of some vertex $P$ looks like $(\R_{\geq 0})^n$ in a neighborhood of the origin in $\R^n$. Then we can identify $L=\hh^0(X_P, \calL)$ with a subset of the ring of rational functions, $\C(X_P)$, 
as described on page \pageref{identification},
 using the section corresponding to the origin as the fixed element $h$:
$$f \mapsto \frac{f}{\textrm{section corresponding to the origin} }.$$

{\bf Notation.} 
For simplicity of notation, given a valuation $\nu$ we will write $\nu(L)$ to denote 
$$\nu(L):=\{\nu(f/h);\ f \in L \setminus \{0\}\}.$$ 
Similarly, let
$\nu(L^m):=\{\nu(f/h^m);\ f \in L^{m} \setminus \{0\}\}$ for any $m >1.$

We denote by $f_j \in \C(X_P)$ the rational function 
coming from the section corresponding to the $j$-th basis vector, $j=1, \ldots, n$. 
Note that $f_1,\ldots,f_n$ form a coordinate system around the fixed point of $X_P$ corresponding to the origin via the moment map. (To see this, one can, for example, use the description of $X_P$ and $f_j$'s from \cite{H}.)

Choose and fix a non-negative integer $c$ and two elements $k<l \in \{1,\dots,n\}$. Then 
\begin{equation}
\label{coor sys}
\begin{split}
u_1&=f_1,\\
&\ldots,\\
u_{k-1}&=f_{k-1},\\
u_{k}&=f_{k}-f_l^c,\\
u_{k+1}&=f_{k+1},\\
&\ldots,\\
u_n&=f_n,
\end{split}
\end{equation}
also forms a coordinate system.
Let $\nu$ be the associated lowest term valuation (as in Example \ref{eg lh val}). 
The image $\nu(L)$ can be obtained by using a ``sliding" operator $\calF_{-e_k+c e_l}$, defined as follows.
For each affine line $\ell$ in $\R^n$ in the direction of $-e_k+c e_l$, with $P \cap \ell \cap Z^n \neq \emptyset$, translate the set $\{P \cap \ell \cap \Z^n\}$ by $a(-e_k+c e_l)$ with $a \geq 0$ maximal non-negative number for which $a(-e_k+c e_l) +\{P \cap \ell \cap \Z^n\}\subset (\R_{\geq 0})^n$.
\begin{lemma}\label{sliding}
One obtains $\nu(L)$ by sliding the integral points of $P$ in the direction $-e_k+c e_l$, inside $(\R_{\geq 0})^n$,
i.e., 
$$\nu (L)=\calF_{-e_k+c e_l}(P \cap \Z^n).$$
\end{lemma}
Instead of the proof, which can be found in \cite{PT}, we give the following  example which illustrates the main idea.
\begin{example}\label{eg slide}
Let $(X_P, \omega_P)$ be the symplectic toric manifold corresponding to the polytope $P=\conv\,\{(0,0),(1,0),(1,3),(0,3)\} \subset \R^2$. That is, $X_P$ is diffeomorphic to $\CP^1 \times \CP^1$ with product symplectic structure (with different rescaling of the Fubini--Study symplectic form on each factor).  
Let $\nu$ be the lowest term valuation associated to the coordinate system $$u_1=f_1-f_2^2,\ u_2=f_2.$$
One can easily compute that $\nu(f_1)=(0,2),\  \nu(f_2)=(0,1),\ \nu(f_1f_2^3)=(0,5),$ and in general 
$$\nu(f_1^a f_2^b)=(0,2a+b),\ a,b \in \Z_{\geq 0}.$$
Futhermore, $\nu (f_1-f_2^2)=(1,0)$, $\nu (f_1f_2-f_2^3)=(1,1),$
and one can observe that
$$\nu(L)=\calF_{(-1,2)}(P \cap \Z^2).$$
The polytopes $P$ and $\conv(\nu(L ))$ are presented on Figure \ref{fig Hirzebruch}.
\end{example}
Understanding $\nu(L)$ is not enough for constructing and understanding a toric degeneration. First of all, to construct a flat family with toric fiber $\pi^{-1}(0)$ one needs the associated semigroup $S=S(\nu)$ to be finitely generated. Additionally, the toric fiber $\pi^{-1}(0)=\Proj \C[S]$ is the toric variety associated to the Okounkov body $\Delta$ if $\Proj \C[S]$ is normal, that is, if $S$ is saturated. Moreover, to describe the Okounkov body one also needs to find $\nu(L^m )$ for $m>1$. Note that in general $L^m$ differs from $\hh^0(X,\calL^{\otimes m})$.
The following proposition describes an especially nice situation where all these conditions simplify. 
\begin{proposition}\label{nice cond give sympl}
Let $(X, \omega=\Phi_\calL^*(\omega_{FS}))$ be a $2n$-dimensional projective symplectic toric manifold associated to a smooth polytope $P$, with the projective embedding induced by a very ample line bundle $\calL$.
Let $\nu$ be a lowest term valuation associated to a coordinate system \eqref{coor sys}, and $S$ the induced semigroup.
Assume that there exists a smooth integral polytope $\Delta$ such that
$$S=(\cone \, \Delta)\cap (\Z \times \Z^n).$$
Then $(X, \omega)$ is symplectomorphic to the symplectic toric manifold $(X_\Delta, \omega_\Delta)$ associated to $\Delta$ via Delzant's construction.
\end{proposition}

\begin{proof}[Sketch of a proof]
The assumptions imply that the semigroup $S$ is saturated and (by Gordan's Lemma) finitely generated. Therefore there is a toric degeneration $(\mathfrak{X}, \tilde{\omega})$ with generic fiber $(X, \omega)$ and the special fiber $\pi^{-1}(0)=\Proj \C[S]$ which is a normal toric variety. Moreover, the Okounkov body associated to the semigroup $S$ is precisely $\Delta$ and therefore $\Proj \C[S]$, equipped with the restriction of  $\tilde{\omega}$, is the symplectic toric manifold $(X_\Delta, \omega_\Delta)$ associated to $\Delta$ via Delzant's construction.
\end{proof}
Note that $S=(\cone \, \Delta)\cap (\Z \times \Z^n)$ implies, in particular, that $\nu (L^m)$ contains ``enough" of integral points, namely that
$$\forall_{m\geq 1}\ \nu(L^m)=m\,\Delta\cap \Z^n=\conv(\nu(L^m))\cap \Z^n.$$

To understand better the requirement $\conv\,(\nu(L^m )) \cap \Z^n=\nu(L^m ),$ consider the following example.
\begin{example} (``Enough" of integral points and saturation.)
Let $(X_P, \omega_P)$ be the symplectic toric manifold corresponding to the polytope $P=\conv\,\{(0,0),(2,0),(2,2),(0,2)\} \subset \R^2$, that is, $X_P$ is diffeomorphic to $\CP^1 \times \CP^1$ as in the previous example, but with a different symplectic form.  
As before, let $\nu$ be the lowest term valuation associated to the coordinate system $$u_1=f_1-f_2^2,\ u_2=f_2.$$
 Then 
 $$\nu(L)=\calF_{(-1,2)}(P \cap \Z^2)=\{(0,j); j=0,\ldots, 6\}\cup \{(1,0), (1,3)\}\subsetneq \conv\,(\nu(L))\cap \Z^2.$$
 The figure below presents the integral points $P\cap \Z^2$ and $\nu(L)=\calF_{(-1,2)}(P \cap \Z^2)$. The semigroup $S$ is not saturated: we have that $(1,1,1) \notin S$ even though $(2,2,2) \in S$, as $(2,2,2)=\nu (f_1(f_1-f_2^2) \cdot (f_1-f_2^2)) \in \{2\}\times \nu (L^2)$.

\begin{figure}[h]
\includegraphics[scale=0.5]{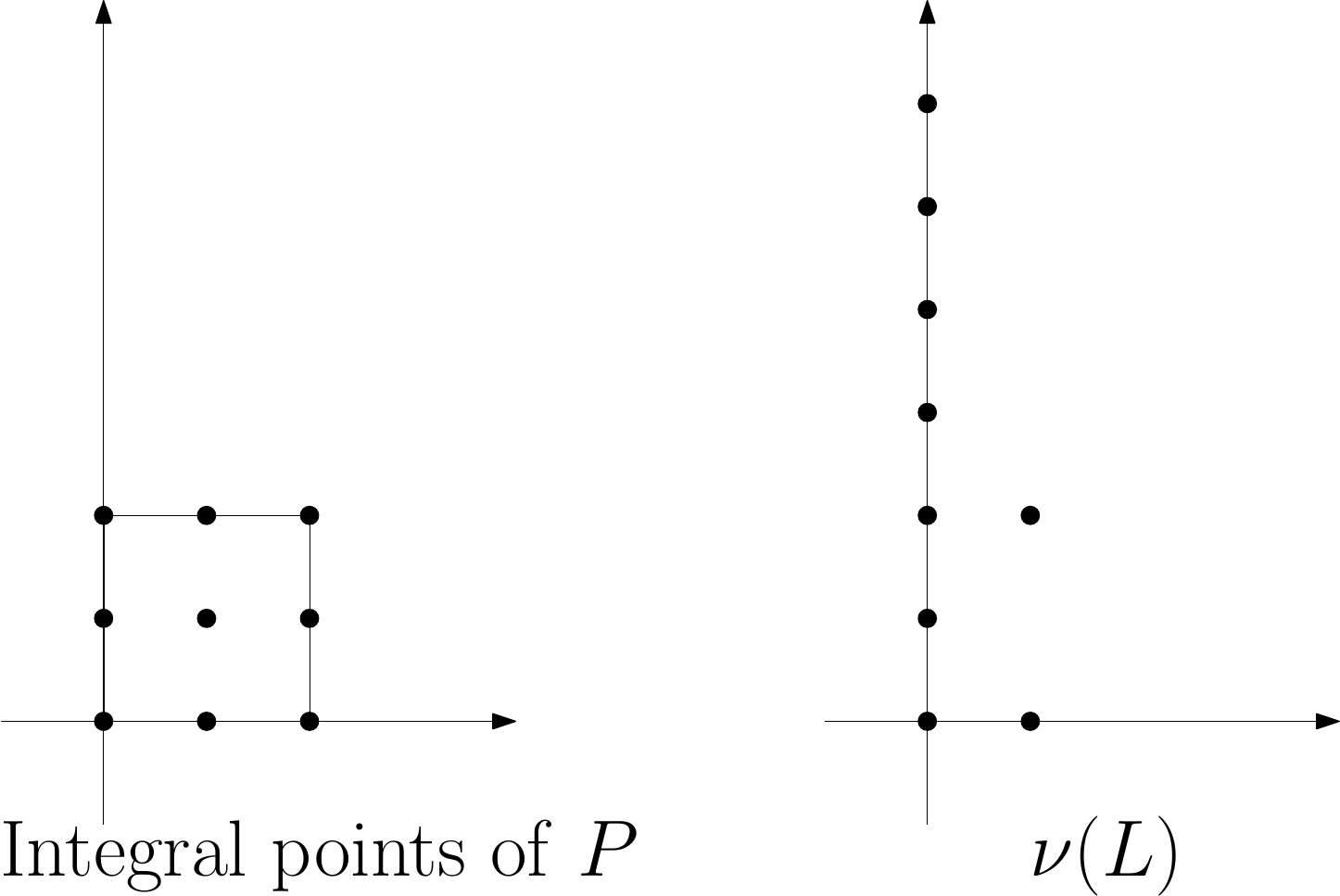}
\end{figure}

 \end{example}

  The following condition is sufficient, though not necessary, to guarantee that we have enough of integral points.
  \begin{corollary}\label{cor int pt}
Let
$$\Delta= \big\{ p \in \R^2 \ \big| \ 0\leq \langle p, e_1\rangle \leq \lambda_1,\  0\leq \langle p, e_2\rangle \mbox{ and }
\big\langle p, e_2 +  A e_1 \big\rangle \leq \lambda_2 \big\}\textrm{ and } c \in \Z_{>0}.$$
If 
$$\lambda_2-c\lambda_1 >0,$$
then 
$$(\conv\, \calF_{(-1,c)}(\Delta \cap \Z^2))\cap \Z^2=\calF_{(-1,c)}(\Delta \cap \Z^2).$$
\end{corollary}
 Note that in that case the polytope $\conv\, \calF_{(-1,c)}(\Delta \cap \Z^2)$ is also a trapezoid, namely:
  $$\big\{ p \in \R^2 \ \big| \ 0\leq \langle p, e_1\rangle \leq \lambda_1,\  0\leq \langle p, e_2\rangle \mbox{ and }
\big\langle p, e_2 +  (2c-A) e_1 \big\rangle \leq \lambda_2+(c-A)\lambda_1) \big\},$$ 
if $c>A$ (see Figure \ref{fig EgDegeneration}), or $\Delta$ if $c\leq A$. 
 \begin{figure}[h]
 \includegraphics[scale=0.8]{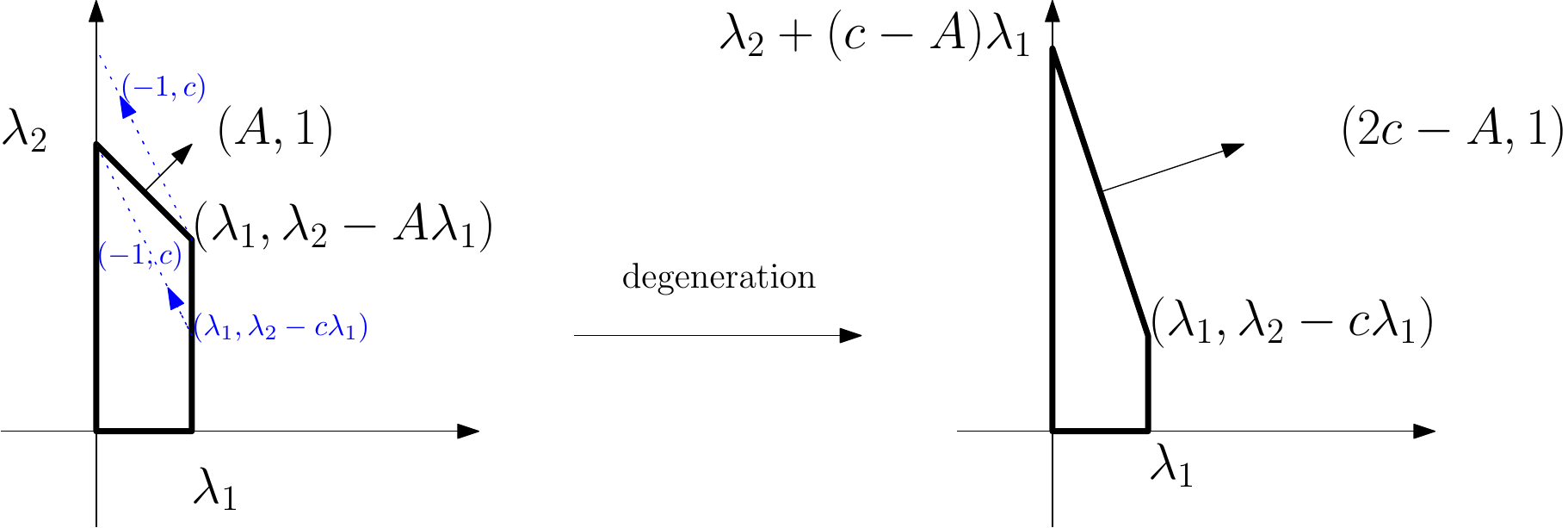}
 \caption{Toric degeneration of a Hirzebruch surface.}\label{fig EgDegeneration}
 \end{figure}

\subsection{Cohomological rigidity for Bott manifolds}\label{sec bott mfd}
A Bott manifold is a manifold obtained as the total space of a tower of iterated bundles with fiber $\CP^1$ and first base space $\CP^1$. 
Such manifold naturally carries an algebraic torus action, and can be viewed as a toric manifold. Note that $4$-dimensional Bott manifolds are exactly the Hirzebruch surfaces discussed in Example \ref{eg h}.
For more information about Bott manifolds see, for example, \cite{GK}.

The simplest example of a $2n$-dimensional Bott manifold is the product of $n$ copies of $\CP^1$. 
Equipped with a product symplectic structure 
$\omega=\pi_1^*(a_1 \omega_{FS} ) + \ldots +\pi_n^*(a_n \omega_{FS} )$, for some $a_j \in \R_{>0}$, and the standard toric action\footnote{
In the description of the symplectic structure, $\pi_j \colon \CP^1 \times \ldots \times \CP^1 \rightarrow \CP^1$ denotes the projection onto the $j$-th factor, and $\omega_{FS}$ stands for the Fubini--Study symplectic form. 
The standard action of $(S^1)^n$ on $(\CP^1)^n$ is the one where each $S^1$  acts on the respective copy of $\CP^1$ by $e^{it}\cdot[(z_0,z_1)]=[(z_0,e^{it}z_1)]$.} it becomes a symplectic toric manifold, whose Delzant polytope is a product of intervals, with lengths depending on $a_j$'s. In particular, if all $a_j$'s are equal, then the moment image is a hypercube.

A moment image for a general $2n$-dimensional Bott manifold is combinatorially an $n$-dimensional hypercube. By applying a translation and a $\gl(n,\Z)$ transformation one can always arrange for the moment image to be a polytope of the form
$$ \Delta=\Delta(A, \lambda) = \Big\{ p \in \R^n \ \big| \ \langle p, e_j\rangle \geq 0 \mbox{ and }
\big\langle p, e_j + \sum_i A^i_j e_i \big\rangle \leq \lambda_j \textrm{ for all } 1 \leq j  \leq n\Big\},$$
where $\lambda \in (\R_{>0})^n$ and the parameters $A^i_j$ satisfy that $A^i_j = 0$ unless $i < j$ and thus can be arranged in an $n \times n$ strictly upper-triangular  integral matrix $A \in \mat$. Certain relation between $A$ and $\lambda$ must be satisfied in order for $\Delta(A, \lambda)$ to have $2^n$ facets and be combinatorially equivalent to a hypercube (see \cite{PT}).
In that case we say that $(A, \lambda)$ {\it defines a symplectic toric Bott manifold} $(M_{A}, \omega_{ \lambda})$ corresponding to the Delzant polytope $\Delta(A, \lambda)$. The matrix $A$ encodes the twisting of consecutive $\CP^1$ bundles, and thus determines a diffeomorphism type of $M_{A}$, while $\lambda$ determines the symplectic structure.
By a classical result of Danilov \cite{D}
\begin{equation}\label{cohomology}
\hh^*(M_A;\Z) =\Z[x_1,\dots,x_n]/\big(x_i^2 + \sum_j A^i_j x_j x_i \big),
\end{equation}
with $[\omega_\lambda]=\sum_i \lambda_i x_i$. Note that this particular presentation of $\hh^*(M_A;\Z)$ depends on $A$. (The element $x_j$ is the Poincar\'e dual to the preimage of the facet $\Delta(A,\lambda)\cap \{\big\langle p, e_j + \sum_i A^i_j e_i \big\rangle = \lambda_j\}$.)
Using the above presentation we define the following special elements  
 \begin{equation}\label{special elements}
 \alpha_k = - \sum_j A^k_j x_j \in \hh^*(M_A;\Z), \quad y_k = x_k -\frac{1}{2} \alpha_k \in \hh^*(M_A;\Q)
 \end{equation}
  for all k.
We say  $x_k$ is of {\it even (respectively odd)
exceptional type} if $\alpha_k = c y_\ell$ for some $\ell > k$, where $c$ is an
even (respectively odd) integer. \label{exceptional def}
In ``coordinates", this means that $A_j^k = 0$ for $j < \ell$ and $A_j^k = \frac{1}{2} A^k_\ell A_j^\ell$ for $j > \ell$.

We say that a Bott manifold is {\it $\Q$-trivial} if  $\hh^*(M;\Q) \simeq \hh^*((\CP^1)^n;\Q)$.
For example, observe that all Hirzebruch surfaces are $\Q$-trivial Bott manifolds.

Using toric degenerations one can prove the following result, which is the key ingredient of the proof of Theorem \ref{thm coh rig}.
\begin{proposition}\cite{PT}\label{basic change}
Let $(M,\omega)$ and $(\tM,\tomega)$ be symplectic Bott manifolds
associated to strictly upper triangular $A$ and $\tA$ in $\mat$
and $\lambda$ and $\tlambda$ in $\Z^n$, respectively.
Assume that there exist integers $1 \leq k < \ell \leq n$
so that $A_\ell^k$ and $ \tA_\ell^k $ are of the same parity and the isomorphism from $\Z[x_1,\dots,x_n]$ to $\Z[\tx_1,\dots,\tx_n]$
that sends $x_k$ to $\tx_k +\frac{\tA_\ell^k - A_\ell^k}{2} \ \tx_\ell$ and $x_i$ to $\tx_i$ for all
$i \neq k$ descends to an isomorphism from $\hh^*(M;\Z)$ to $\hh^*(\tM;\Z)$
and takes $\sum \lambda_i x_i$ to $\sum \tlambda_i \tx_i$.
If $A_\ell^k + \tA_\ell^k \geq 0$, then
$M$ and $\tM$ are symplectomorphic.
\end{proposition}
\begin{proof}[Sketch of a proof]
Without loss of generality we can assume that the polytope $\Delta(A,\lambda)$ associated to $(A, \lambda)$ is {\it normal}, that is,
any integral point of $m\,\Delta(A,\lambda)$ can be expressed as a sum of $m$ integral points of $\Delta(A,\lambda)$:
$$\forall_{m \in \Z_{>0}}\ \forall_{x \in m\,\Delta(A,\lambda)\cap \Z^n}\ \exists_{x_1,\ldots,x_m \in \Delta(A,\lambda)\cap \Z^n}\textrm{  such that  } x=x_1+\ldots+x_m.$$ 
Indeed, if $\Delta(A,\lambda)$ is not a normal polytope, replace  $(M,\omega)$ and $(\tM,\tomega)$ by  $(M,(n-1)\,\omega)$ and $(\tM,(n-1)\,\tomega)$. This dialates the corresponding polytopes by $(n-1)$.
For any integral polytope $P \subset \R^n$ its  dialate $m P$ with $m \geq n-1$ is normal (see, for example, \cite[Theorem 2.2.12]{CLS}). Obviously if $(M,(n-1)\,\omega)$ and $(\tM,(n-1)\,\tomega)$ are symplectomorphic, then so are $(M,\omega)$ and $(\tM,\tomega)$.
As usually, let $\calL$ denote the very ample line bundle over $M$ corresponding to $\omega$, and $L$ the space of its holomorphic sections.
Note that normality implies that $L^m$ can be identified with $H^0(M, \calL^{\otimes m})$  as a basis for both of these vector spaces is given by the integral points $m\,\Delta(A,\lambda)\cap \Z^n$.

Also without loss of generality we can assume that $\tA_\ell^k\geq A_\ell^k$.
Let $c=\frac{1}{2}(A_\ell^k + \tA_\ell^k) \geq 0$.
We will work with a lowest term valuation $\nu$ associated to the coordinate system \eqref{coor sys}.
From Lemma \ref{sliding} and the normality assumption, for all $m \geq 1$ we have that
$$\nu(L^m)=\mathcal{F}_{-e_k+ce_l}(m\,\Delta(A,\lambda)\cap \Z^n).$$

To understand $\nu(L^m)$ consider the action of $\mathcal{F}_{-e_k+ce_l}$ on  $2$-dimensional ``slices", that is, the intersections of $m\,\Delta(A,\lambda)$ with affine subspaces which are translations of $(e_k,e_l)$-planes. Such slices are either empty or are trapezoids like in Example \ref{eg slide} and Corollary \ref{cor int pt}, possibly with a cut.
A slightly tedious computation shows that 
$$\forall_{m \geq 1}\ \conv(\nu(L^m))=m\,\Delta(\tA, \tlambda).$$
For that computation one uses relations between $A$, $\lambda$, $\tA$ and $\tlambda$ which are implied by the
 facts that $\Delta(A,\lambda)$ and $\Delta(\tA, \tlambda)$ are combinatorially hypercubes, and by the existence of the isomorphism described in the statement of the proposition. 
 These relations also allow to generalize Corollary \ref{cor int pt} (precisely: to show that the appropriate generalization of condition $\lambda_2-c\lambda_1>0$ holds) and show that
$$\nu(L^m)=m\,\Delta(\tA, \tlambda)\cap \Z^n.$$
This means that the semigroup $S$ associated to the valuation $\nu$ of $(M,\omega)$ is exactly $S=(\cone \,\Delta(\tA, \tlambda))\cap (\Z \times \Z^n)$.
 Then the claim follows from Proposition \ref{nice cond give sympl}.
\end{proof}
Note that if $x_k$ is even (resp. odd) exceptional, say $\alpha_k=my_l$, then one can construct an isomoprhism as in Proposition \ref{basic change} from 
$\hh^*(M_A;\Z)$ to 
$\hh^*(M_{\tA};\Z)$ for some $\tA$ with $\tA_\ell^k$ equal to $0$ (resp. $-1$).
For example if $x_k$ is of even exceptional type, i.e., $\alpha_k=2my_l$ for some $m$ and $\ell$, implying that  $A^k_{\ell}=-2m$ and $A^k_j=-mA^\ell_j$ for $j\neq \ell$, then one should put $\tA^k_l=0$, $\tA^i_j=A^i_j$ for all $i$ and all $j \neq \ell$, and $\tA^i_l=A^i_\ell+mA^i_k$ for all $i \neq k$. 
Therefore, consecutive applications of the above proposition lead to simplifying the description of a given Bott manifold.

\begin{corollary}
Any symplectic toric Bott manifold is symplectomorphic to one for which $A^i_j=0$ (resp. $A^i_j=-1$) whenever $x_i$ has even (resp. odd) exceptional type and $\alpha_i=my_j$.
\end{corollary}

In the case of $\Q$-trivial Bott manifolds all 
$x_i$ have exceptional type as shown in \cite[Proposition 3.1]{CM}. 
Therefore, such symplectic toric Bott manifold must be a product of the following standard models of $\Q$-trivial Bott manifolds.

\begin{example}($\Q$-trivial Bott manifolds)
Take $n \in \Z_{>0}$. Let $A_n^i = -1$ for all $1\leq i < n$, 
and $A^i_j = 0$ otherwise. For such upper triangular matrix $A=[A_j^i]$
and any $\lambda \in (\R_{>0})^n$, the polytope $\Delta(A, \lambda)$
 is combinatorially a hypercube, thus it defines a symplectic toric Bott manifold, which we will denote by $\mathcal{H}=\mathcal{H}(\lambda_1,\ldots,\lambda_n)$.
Observe that 
$$\hh^*(\mathcal{H};\Z) = \Z[x_1,\dots,x_n]/\big( x_1^2 - x_1x_n, \dots, x_{n-1}^2 - x_{n-1} x_n, x_n^2 \big).$$
Consider elements $y_i \in \hh^*(\mathcal{H};\Q)$ defined by
$y_i = x_i - \frac{1}{2} x_n$ for all $i < n$, and $y_n = x_n$, and note that they form a basis for $\hh^*(\mathcal{H};\Q)$.
Moreover, as $y_i^2 = 0$ for all $i$, we get that $\hh^*(\mathcal{H};\Q) \simeq \Q[y_1,\dots,y_n]/\big(y_1^2, \ldots, y_n^2)$, that is, $\mathcal{H}$ is $\Q$-trivial.

More generally, any partition of $n$, $\sum_{i=1}^m l_i = n$ together with $\lambda \in (\R_{>0})^n$, define a $\Q$-trivial Bott manifold
$$\mathcal{H}(\lambda_1, \ldots, \lambda_{l_1})\,\times \ldots \times \mathcal{H}(\lambda_{n-l_m+1}, \ldots, \lambda_{n}).$$
\end{example}
\begin{corollary}\label{std form}
Each $2n$-dimensional $\Q$-trivial Bott manifold $M$ with integral symplectic form is symplectomorphic to
$$\mathcal{H}(\lambda_1, \ldots, \lambda_{l_1})\,\times \cdots \times \mathcal{H}(\lambda_{n-l_m+1}, \ldots, \lambda_{n}),$$
for some partition $n=\sum_{i=1}^m l_i$ of $n$ and some $\lambda_1, \ldots, \lambda_n\in \Z_{>0}$.
\end{corollary}
The above standard model is easy enough, so that one can understand all possible ring isomorphisms between cohomology rings and prove that they are induced by maps on manifolds.

\begin{lemma}\label{str rigidity hn}
Fix $n \in \Z_{>0}$.
Let $\sum_{i=1}^m l_i = \sum_{i = 1}^{\wt{m}} \wt{l}_i = n$
be  partitions  of $n$,
and let $\lambda, \tlambda \in (\R_{>0})^n$.
Consider symplectic Bott manifolds
\begin{gather*}
(M,\omega)= \mathcal{H}(\lambda_1,\dots,\lambda_{l_1} )
\times \cdots \times \mathcal{H}(\lambda_{n - l_m + 1},\dots,\lambda_{n}); \\
(\tM,\tomega)= \mathcal{H}(\tlambda_1,\dots,\tlambda_{\wt{l}_1} )
\times \cdots \times \mathcal{H}(\tlambda_{n - \wt{l}_{\wt{m}} + 1},\dots,\tlambda_{n}).
\end{gather*}
Given a ring isomorphism $F \colon \hh^*(M;\Z) \to \hh^*(\tM;\Z)$ 
such that $F[\omega] = [\tomega]$,
there exists a symplectomorphism $f$  from  $(\tM,\tomega)$ to
$(M,\omega)$ so that $\hh^*(f) = F$.
\end{lemma}
\begin{proof}[Sketch of a proof]
First consider the situation when
$$(M,\omega)= \mathcal{H}(\lambda_1,\dots,\lambda_{n}) 
\quad \mbox{and} \quad 
(\tM,\tomega)= \mathcal{H}(\tlambda_1,\dots,\tlambda_{n}).$$
The $\Q$-triviality assumption implies that there are exactly $2n$ primitive
classes in $\hh^2(M;\Z)$ which square to $0$. 
A short computation shows that these are $\pm z_1,\dots, \pm z_n$,
where $z_n = x_n$ and $z_i = 2 x_i - x_n$ for all $i < n$. 
Similarly for $\tM$. Any  ring isomorphism between $\hh^*(M;\Z)$ and  $ \hh^*(\tM;\Z)$ must restrict to a bijection on the set of such elements, that is, there exists $\epsilon=(\epsilon_1, \ldots, \epsilon_n)\in \{-1,1\}^n$ and a permutation $\sigma \in \mathcal{S}_n$ such that $F(z_j)=\epsilon_j \tz_{\sigma(j)}$.
Moreover, presenting $[\omega]$ (resp. $[\tomega]$) in the basis $\{z_1,\ldots,z_n\}$ (resp. $\{\tz_1,\ldots,\tz_n\}$) and recalling that the isomorphism $F$ maps 
$[\omega]$ to $[\tomega]$, one can deduce that $F$ acts by a permutation: $F(z_j)=\tz_{\sigma(j)}$ for some permutation $\sigma \in \mathcal{S}_n$ with $\sigma(n)=n$, and that $\lambda_j=\tlambda_{\sigma(j)}$. 
Then $F$ also takes $x_i$ to $x_{\sigma(i)}$ and it holds that
$ A^i_j=\tA^{\sigma(i)}_{\sigma(j)}$ for all $i,j.$
If $\Lambda \in \gl(n,\Z)$ denotes the unimodular matrix taking $e_i$ to $e_{\sigma(i)}$, 
then $\Lambda^T( \Delta(\tA,\tlambda)) = \Delta( A,  \lambda)$. Therefore, by the Delzant theorem, the manifolds $(M,\omega)$ and $(\tM,\tomega)$ are (equivariantly) symplectomorphic, by some symplectomorphism $f$. Moreover, as $\Lambda^T$ maps the facet $\{  \langle p, e_{\sigma(j)} \rangle = 0 \} \cap \Delta(\tA,\tlambda)$ to the facet $\{ \langle p, e_j \rangle = 0 \} \cap \Delta(A,\lambda)$,
 and 
$\{  \langle p, e_{\sigma(j)} + \sum_i \tA_{\sigma(j)}^i e_i \rangle = \tlambda_{\sigma(j)} \} \cap \Delta(\tA,\tlambda)$ to  $\{ \langle p, e_j  + \sum_i A_j^i e_i \rangle = \lambda_j \} \cap \Delta(A,\lambda)$,
the map $\hh^*(f)$ induced by $f$ on cohomology maps the Poincar\'e duals of preimages of these facets accordingly. That is, $\hh^*(f)=F$.

In the general case, denote by $\lambda^{l_s}$ the $l_s$-tuple of
numbers $(\lambda_{l_1+\cdots+l_{s-1}+1},\dots, \lambda_{l_1 +\cdots +l_s}),$ and define $\tlambda^{\tl_s}$ similarly.
Again, we look at primitive elements with trivial squares. In $\hh^*(M;\Z)$ these are precisely
$$\pm x_{l_s} \mbox{ and } \pm(2x_i-x_{l_s}) \textrm{ for }
s=1,\ldots,m
\mbox{ and } 
i_{s-1}< i <i_s .$$
Note that each such element is contained in some subring $\hh^*(\mathcal{H}(\lambda^{l_s});\Z)\subseteq \hh^*(M;\Z)$, and that all primitive elements in $\hh^*(\mathcal{H}(\lambda^{l_s});\Z)$ whose square is zero  
are equal modulo $2$. Therefore $F$ restricts to an isomorphism from $\hh^*(\mathcal{H}(\lambda^{l_s});\Z)$ to $\hh^*(\mathcal{H}(\tlambda^{\tl_r});\Z)$ for some $r$ with $l_s=\tl_r$. This implies that both partitions of $n$ must be equal up to permutation of their factors. Repeating the arguments of the previous paragraph, one can construct a symplectomorphism inducing the ring isomorphism $F$.
\end{proof}
\begin{proof}[Proof of Theorem~\ref{thm coh rig} ]
Theorem~\ref{thm coh rig} follows immediately from Corollary \ref{std form} and Lemma \ref{str rigidity hn}.
\end{proof}

\end{document}